\documentclass[11p]{amsart}
\usepackage{amsmath}
\usepackage{amsfonts}
\usepackage{amssymb}

\usepackage{amsthm}
\usepackage[mathcal]{euscript}
\usepackage{amscd}
\usepackage[OT2,OT1]{fontenc}

\newcommand{\lra}{\longrightarrow}

\newcommand{\bP}{\mathbb{P}}

\newcommand{\bC}{\mathbb{C}}

\newcommand{\PP}{\mathbb{P}}

\newcommand{\Kb}{{\overline K}}
\newcommand{\bK}{{\overline K}}
\newcommand{\Kbv}{{\overline K}_v}
\newcommand{\bKv}{{\overline K}_v}

\newcommand{\hhat}{\hat{h}}

\newcommand{\cyr}{%
        \renewcommand{\rmdefault}{wncyr}%
        \renewcommand{\sfdefault}{wncyss}%
        \renewcommand{\encodingdefault}{OT2}%
        \normalfont
        \selectfont}
\DeclareTextFontCommand{\textcyr}{\cyr}

\DeclareMathOperator{\Gal}{Gal}

\newcommand{\Kbar}{\,\overline{\! K}}

\newcommand{\PNhyp}{(P,N,v,\infty)}

\newcounter{rmenum}

\newcounter{alphenum}

\newtheorem{thm}{Theorem}[subsection]

\newtheorem{lem}[thm]{Lemma}

\newtheorem{prop}[thm]{Proposition}
\newtheorem{conj}[thm]{Conjecture}

\theoremstyle{definition}

\newtheorem*{defn}{Definition}

\begin{document}


\newpage
%


\thispagestyle{empty}

\title[A finiteness property for Chebyshev polynomials]{A finiteness
    property for preperiodic points of Chebyshev polynomials}
 
\author{Su-Ion Ih} 
\address{Su-Ion Ih\\ 
Department of Mathematics \\ 
University of Colorado at Boulder \\ 
Campus Box 395 \\ 
Boulder, CO 80309-0395 \\ 
USA} 
\email{ih@math.colorado.edu}

\author{Thomas Tucker} 
\address{Thomas Tucker\\ 
Department of Mathematics\\ 
University of Rochester\\
Rochester, NY 14627}
\email{ttucker@math.rochester.edu} 
 


\thanks{The second author was partially supported by NSA
    Grant 06G-067.}

\begin{abstract} 
  Let $K$ be a number field with algebraic closure $\overline K$, let
  $S$ be a finite set of places of $K$ containing the archimedean
  places, and let $\varphi$ be a Chebyshev polynomial.  We prove that
  if $\alpha \in \overline K$ is not preperiodic, then there
  are only finitely many preperiodic points $\beta \in \overline K$
  which are $S$-integral with respect to $\alpha$.
\end{abstract} 
\maketitle

\section{Introduction}

Let $K$ be a number field with algebraic closure $\overline K$, let
$S$ be a finite set of places of $K$ containing the archimedean
places, and let $\alpha, \beta \in \bK$.  We say that $\beta$ is
$S$-integral relative to $\alpha$ if no conjugate of $\beta$ meets any
conjugate of $\alpha$ at primes lying outside of $S$.  More precisely,
this means that for any prime $v \notin S$ and any $K$-embeddings
$\sigma:K(\alpha) \lra {\overline K_v}$ and $\tau: K(\alpha) \lra
{\overline K_v}$, we have
\begin{equation*} 
\left\{ \begin{array}{ll} 
|\sigma(\beta)-\tau(\alpha)|_v \ge 1 & \text{if $|\tau(\alpha)|_v \le 1$\ ; \text{and} } \\ 
|\sigma(\beta)|_v \le 1 & \text{if $|\tau(\alpha)|_v > 1$\ .} 
       \end{array} \right. 
\end{equation*} 
Note that this definition extends naturally to the case where $\alpha$
is the point at infinity.  We say that $\beta$ is $S$-integral
relative to the point at infinity if $|\sigma(\beta)|_v \leq 1$ for
all $v \notin S$ and all $K$-embeddings $\sigma:K(\beta) \lra {\overline
  K}_v$.  Thus, our $S$-integral points coincide with the usual
$S$-integers when $\alpha$ is the point at infinity.  \vskip .05 in

In \cite{BIR}, the following conjecture is made.
\begin{conj}[Ih]\label{Ih conj}
  Let $K$ be a number field, and 
  let $S$ be a finite set of places of $K$ that contains all the
  archimedean places. If $\varphi: \bP^1_K \lra \bP^1_K$ is a
  nonconstant rational function of degree $d > 1$ and $\alpha \in
  \PP^1(K)$ is non-preperiodic for $\varphi$, then there are at most
  finitely many preperiodic points $\beta \in \PP^1(\Kbar)$ that are
  $S$-integral with respect to $\alpha$.
\end{conj}

In \cite{BIR}, it is proved that this conjecture holds when $\varphi$
is a multiplication-by-$n$ (for $n \geq 2$) map on an $\mathbb G_m$ or
on an elliptic curve.  Recently, Petsche \cite{clay} has proved the
conjecture in the case where the point $\alpha$ is in the $v$-adic
Fatou set at every place of $K$.  A similar problem, dealing with
points in inverse images of a single point rather than with
preperiodic points, has been treated by Sookdeo \cite{vijay}.

In this paper, we show that this conjecture is true for Chebyshev
polynomials. That is, we prove the following, where we note that
$\alpha$ may lie on the Julia set.


\begin{thm}\label{main} Let $\varphi$ be a Chebyshev polynomial. 
  Let $K$ be a number field, and let $S$ be a finite set of places of
  $K$, containing all the archimedean places. Suppose that $\alpha \in K$
  is not of type $\zeta + \zeta^{-1}$ for any root of unity $\zeta$.
  Then the following set
$$ {\mathbb A^1}_{\varphi, \alpha, S} :=
\{ x \in {\overline {\mathbb Q}} : x \; 
\textup{{\em{is $S$-integral with respect to}}}
\; \alpha \; \textup{{\em{and is}}} \; 
\textup{${\varphi}$-{\em{preperiodic}}} \}
$$
is finite.
\end{thm}

This will follow easily from the following theorem.

\begin{thm}\label{comp}
Let $( x_n )_{n=1}^\infty$ be a nonrepeating sequence of preperiodic
points for a Chebyshev polynomial $\varphi$.  Then for any non-preperiodic $\alpha$
in a number field $K$ and any place $v$ of $K$, we have
\begin{equation}\label{v}
\hhat_v(\alpha) =  \lim_{n \to \infty} \frac{1}{ [ K(x_n) : K ] }  \sum_{\sigma:
  K(x_n)/K \hookrightarrow \overline{K}_v } \log | \sigma (x_n) - \alpha
|_v,
\end{equation}
where $\sigma: K(x_n)/K \hookrightarrow \overline K_v$ means
that $\sigma$ is an embedding of $K(x_n)$ into $\overline K_v$,
fixing $K$, here and in what follows.
\end{thm}

Indeed, we will prove Theorem~\ref{comp} slightly more generally
for any $\alpha \in K$ if $v \not | \infty$, while for any $\alpha \neq$
$-2$, 0, or 2 if $v | \infty$. (Note that the proof of Proposition~\ref{most}
actually works for any $\alpha \in [-2, 2] - \{-2, 0, 2 \}$.)  
The proof of Theorem~\ref{main} is then similar to the proof for 
$\mathbb G_m$ given in \cite{BIR}.
Specifically, the proof of Theorem~\ref{comp} breaks down into various
cases, depending on whether or not the place $v$ is finite or infinite
and whether or not the point $\alpha$ is in the Julia set at $v$.  The
fact that the invariant measure for Chebyshev polynomials is not
uniform on $[-2,2]$ provides a slight twist.

The proof of Theorem~\ref{comp} is fairly simple when $v$ is
nonarchimedean.  Likewise, when $v$ is archimedean but $\alpha$ is not
in the Julia set at $v$, the proof follows almost immediately from
an equidistribution result for continuous functions (see
\cite{bilu}).  When $v$ is archimedean and $\alpha$ is in the Julia
set at $v$, however, the proof becomes quite a bit more difficult.  In
particular, it is necessary to use A.~Baker's theorem on linear forms
in logarithms (see \cite{Baker}).  We note that in all cases, our
techniques are similar to those of \cite{BIR}.  

The derivation of Theorem~\ref{main} from Theorem~\ref{comp} goes as
follows: suppose, for contrary, that Theorem~\ref{main} were to be
false. Then we may further assume that the sequence
$(x_n)_{n=1}^\infty$ in Theorem\ref{comp} is a sequence of
$\alpha$-integral points. Then we have

\begin{eqnarray}
0
\ & < & \ 
\widehat h(\alpha) 
\nonumber\\
\ & = & \ 
\sum_{\text{places} \ v \ \text{of} \ K } 
\widehat h_v (\alpha) 
\nonumber\\
\ & = & \ 
\sum_{\text{places} \ v \ \text{of} \ K } 
\lim_{n \to \infty} \frac{1}{ [ K(x_n) : K ] }  
\sum_{\sigma: K(x_n)/K \hookrightarrow \overline{K}_v } \log | \sigma (x_n) - \alpha |_v
\nonumber\\
\ & = & 
\lim_{n \to \infty} \frac{1}{ [ K(x_n) : K ] }  
\sum_{\text{places} \ v \ \text{of} \ K } 
\sum_{\sigma:
  K(x_n)/K \hookrightarrow \overline{K}_v } \log | \sigma (x_n) - \alpha
|_v
\nonumber\\
\ & = & \ 
0,
\nonumber
\end{eqnarray}
where the equality on the third line comes from Theorem~\ref{comp},
the integrality hypothesis on the $x_n$ enables us to switch
$\sum_{\text{places} \ v \ \text{of} \ K }$ and $\lim_{n \to \infty}$
to get the equality on the fourth line, and the last equality is
immediate from the product formula. This is a contradiction.

\section{Preliminaries}

\subsection{The Chebyshev polynomials}

\begin{defn}
 $$
 P_1 (z) := z, \;\;\; P_2 (z) := z^2 - 2; \;\; \textup{and}
 $$ 
 $$
 P_{m+1} (z) + P_{m-1} (z) = z P_{m} (z) \;\; \textup{for all} \;
 m \geq 2. 
 $$
Then a \emph{Chebyshev polynomial} is defined to be any of the $P_{m}$
($m \geq 2$).
\end{defn}

These polynomials satisfy the following properties (see \cite[Section
7]{Milnor}).

\begin{enumerate}

\item For any $m \geq 1$, 
$P_m (\omega +\omega^{-1}) =  \omega^m +\omega^{-m}$, equivalently
$P_m (2 \cos \theta) = 2 \cos (m \theta)$,
where $\omega \in \mathbb C^{\times}$ and $\theta \in \mathbb R$. 

\item For any $\ell, m \geq 1$, $P_\ell \circ P_m =
  P_{\ell m}$.

\item For any $m \geq 3$, $P_m$ has $m-1$ distinct critical
points in the finite plane, but only two critical values, i.e., $\pm 2$. 

\end{enumerate}


\subsection{The dynamical systems of Chebyshev polynomials}

\begin{defn}
Let $\varphi$ be a Chebyshev polynomial. 
The dynamical system induced by
$\varphi$ on ${\mathbb P}^1$ (or ${\mathbb A}^1$) is 
called the \emph{(Chebyshev) dynamical system} with respect to $\varphi$
or the \emph{$\varphi$-dynamical system}. If $\varphi$ is clearly
understood from the context, we simply call it a  
\emph{Chebyshev dynamical system} without reference to $\varphi$.
\end{defn}

\begin{prop}
For any Chebyshev polynomial $\varphi$,
the Julia set of the dynamical system induced by
$\varphi$ (resp.~$- \varphi$) is \textup{[$-2$, 2]}, which is
naturally identified as a subset of the real line on 
the complex plane.
\end{prop}

\noindent {\bf {Proof.}} See \cite[Section 7]{Milnor}.

\vspace{0.15cm}

\begin{prop}~\label{prop;preper} 
  Let $\varphi$ be a Chebyshev polynomial.  Then the finite
  preperiodic points of the $\varphi$-dynamical system are 
  the elements of $\Kb$ of the form $\zeta + \zeta^{-1}$, 
  where $\zeta$ is a root of
  unity.
\end{prop}
\begin{proof}
  Take an element $z \in \Kb$.  Then there is some $a \in \Kb$ such
  that $z = a + \frac{1}{a}$, as can be seen by finding $a$ such that
  $a^2 - az + 1 = 0$.  Note that $a$ cannot be zero. Now if $a$ is not
  a root of unity, then there is some place $w$ of $K(a)$ such that
  $|a|_w > 1$.  Thus, letting $m = \deg \varphi$ $(\geq 2)$, we have
$$ |\varphi^k(z)|_w = 
\left| a^{m^k} + \frac{1}{a^{m^k}} \right|_w > |a|_w^{m^k} - 1,$$
so $|\varphi^k(z)|_w$ goes to infinity as $k \to \infty$.  Hence $z$
cannot be preperiodic.  

Conversely, if $z = \zeta + \zeta^{-1}$, where $\zeta$ is a root of
unity then there are some positive integers $j \not= k$ such that
$\zeta^{m^k} = \zeta^{m^j}$, which gives
$$ \varphi^k(z) = \zeta^{m^k} + \frac{1}{\zeta^{m^k}} = \varphi^j(z),$$
so $z$ is preperiodic for $\varphi$.  
\end{proof}

\subsection{The canonical height attached to a dynamical system}

Let $\varphi$ be a Chebyshev polynomial of degree $m$ and let $v$ be a
place of a number field $K$.  We define the local canonical height
$\hhat_v(\alpha)$ of a point $\alpha \in \Kbv$ associated to
$\varphi$ at any place $v$ of $K$ as
\begin{equation}\label{local}
\hhat_v(\alpha) = \lim_{k \to \infty} \frac{\log \max
  (|\varphi^k(\alpha)|_v, 1)}{m^k}.
\end{equation}

\noindent  This local canonical height has the property that
$$
\hhat_v(\varphi(\alpha)) = m \hhat_v(\alpha)$$
for any $\alpha \in
\Kbv$ (see \cite{CG} for details).  Note that if $v$ is a
nonarchimedean place, then the Chebyshev dynamical system has good
reduction at $v$ and we have $\hhat_v(\alpha) = \log \max ( |\alpha|_v, 1)$.

When $\alpha \in \Kb$, we have
\begin{equation}\label{global}
\hhat(\alpha) = \sum_{\text{places $v$ of $K$}} \hhat_v(\alpha),
\end{equation}
where the left-hand side is the (global) canonical height of $\alpha$
associated to $\varphi$.

In the case of the places $v \mid \infty$, we will use the
$\varphi$-invariant measure $\mu_v := \mu_{v, \varphi}$ (see \cite{L}) for
$\varphi$ to calculate these local heights.  It is worth noticing this
is not a uniform measure on $[-2, 2]$, unlike in the case of the
dynamical system on ${\mathbb P}^1$ with respect to the map $z \mapsto
z^2$, in which case the measure at archimedean places is the uniform
probability Haar measure on the unit circle centered at the origin
(see \cite{bilu}).  The measure has more mass toward the end/boundary
points $\pm 2$ of the Julia set $[-2, 2]$.  Further, the kernel ${
  \frac{1}{\pi} } { \frac{1}{ \sqrt {4-x^2} } }$ has singularities at
the extreme points $\pm 2$.
  
When $v | \infty$, we have the following formula 
for the local height
at $v$ (see \cite[Appendix B]{PST} or \cite{FR2}) for any $\alpha \in \bC$:
\begin{equation}\label{gen}
\hhat_{v} (\alpha)= \int_{\bC} \log | z
- \alpha |_v \; d \mu (z),
\end{equation} 
where $\mu := \mu_{v, \varphi}$ is the unique $\varphi$-invariant measure
with support on the Julia
set of $\varphi$ at $v$.    

Since any root of unity $\xi_k$, say $e^{2 \pi i/k}$, is preperiodic
for the the map sending $z$ to $z^m$, we see that $\xi_k + \xi_k^{-1} = 2
\cos(2\pi/k)$ is preperiodic for $\varphi$.  Now, the preperiodic points
of $\varphi$ are equidistributed with respect to $\mu$ (see \cite{L, BH}),
so for any continuous function $f$ on $[-2,2]$ we have
$$
\lim_{k \to \infty} \frac{1}{k} \sum_{j=1}^k f(2 \cos(j \pi /k)) = \int_\bC f
\, d \mu.$$
Thus $d \mu$ is the push-forward of the the uniform distribution on
$[0, \pi]$ under the map $\theta \mapsto 2 \cos \theta$, thus
$$ d \mu(x) = \frac{1}{\pi} \frac{d}{dx} \cos^{-1} (x/2) \, dx =
\frac{1}{\pi} \frac{1}{\sqrt{4 - x^2}} \, dx.$$
Thus, \eqref{gen} becomes
\begin{equation}\label{from-PST}
\hhat_{v} (\alpha) = {
  \frac{1}{\pi} } \int_{-2}^{2} { \frac{1}{ \sqrt {4-x^2} } } \log | x
- \alpha |_v \; dx
\end{equation} 
for any $\alpha \in \bC$.

\section{Archimedean places}

\subsection{A counting lemma}

Let $K$ be a number field, and 
let $I \subset [-2, 2]$ be an interval. For any root of unity
$\zeta \in \overline K$, write $x_{\zeta} := \zeta + \zeta^{-1}$. Let
$$
\mathcal N ( x_{\zeta}, I ) \; := \; 
\# \{ \sigma (x_{\zeta}) \in I : \sigma \in 
\Gal \big ( K (x_{\zeta})/K \big ) \}.
$$

\vspace{0.15cm}

\begin{lem}~\label{lem;counting}
Keep notation just above. Let $-2 \leq c < d \leq 2$, and let
$I := (c, d]$ be an interval.
Then for any real $0 < \gamma < 1$ and 
any root of unity $\zeta \in \overline K$,
\begin{equation}\label{first}
\mathcal N ( x_{\zeta}, I )
\; = \;
{ \frac{[ K(x_{\zeta}) : K ]}{\pi} } 
\Big ( \cos^{-1} { \frac{c}{2} } - \cos^{-1} { \frac{d}{2} } \Big )
\; + \; O_{\gamma} \big ([ K(x_{\zeta}) : K ]^{\gamma} \big ) 
\end{equation}
where $\cos^{-1}: [-1, 1] \rightarrow [0, \pi]$ is the $\arccos$
function.  In particular, when $-2 <  c < d < 2$, we may write
\begin{equation}\label{M}
\mathcal N ( x_{\zeta}, I ) \leq  M [ K(x_{\zeta}) : K ](d-c)
+ O_{\gamma} \big ([ K(x_{\zeta}) : K ]^{\gamma} \big )  
\end{equation}
where $M := M_{c,d}$ 
is the supremum of $\frac{1}{\sqrt{4 - x^2}}$ on $(c, d]$.    
\end{lem}

\begin{proof}
Write $\zeta = e^{2\pi i { \frac{a}{N} } }$, where $N$ is a positive
integer and $1 \leq a \leq N$. Then note 
\begin{eqnarray}
x_{\zeta} \in I
\; &\Longleftrightarrow& \;
e^{2\pi i { \frac{a}{N} } } + e^{-2\pi i { \frac{a}{N} } } \in I
\nonumber\\
\; &\Longleftrightarrow& \;
\cos \Big ( 2 \pi { \frac{a}{N} } \Big ) \in 
\Big ( { \frac{c}{2} }, { \frac{d}{2} } \Big ] \nonumber\\
\; &\Longleftrightarrow& \;
a \in { \frac{N}{2 \pi} } 
\Big [ \cos^{-1} { \frac{d}{2} },  \cos^{-1} { \frac{c}{2} } \Big ).
\nonumber
\end{eqnarray}
Then \eqref{first} follows immediately from
\cite[Prop. 1.3]{BIR}.  To see that \eqref{M} holds, note that the
derivative of the function $\cos^{-1} (x/2)$ is $\frac{1}{\sqrt{4 -
    x^2}}$.  Thus, \eqref{M} is a consequence of \eqref{first} and
along with the Mean Value Theorem from calculus.
\end{proof}

Remark. In the above, more precisely, we may define $M$ to be 
the supremum of $\frac{1}{\pi} \frac{1}{\sqrt{4 - x^2}}$ on $(c, d]$.       
However, this difference will not matter for our later purpose. So
we will keep the above choice for $M$. 

\subsection{Baker's lower bounds for linear forms in logarithms}

Here we state the theorem on 
Baker's lower bounds for linear forms in logarithms, 
(see \cite{Baker}[A.~Baker, Thm. 3.1, p. 22]).
\begin{thm}[Baker]
\label{thm;baker}
Suppose that 
$e^{2 \pi i \theta_0} \in \overline {\mathbb Q}$. 
Then there exists a constant $C := C(\theta_0) > 0$ such that
for any coprime $a, N \in \mathbb Z$ ($N \neq 0$ or $\pm 1$)
with ${ \frac{a}{N} } \neq \theta_0$,
$$
\Big | { \frac{a}{N} } - \theta_0 \Big | \; \geq \;
M^{-C}
$$ 
where $M := \max ( |a|, |N| )$. 
\end{thm}

\begin{proof}
\begin{eqnarray}
\Big | { \frac{a}{N} } - \theta_0 \Big | 
\; & = & \; 
{ \frac{1}{2 \pi} }
\Big | { \frac{a}{N} } \cdot 2 \pi i - 2 \pi i \theta_0 \Big |. \nonumber
\end{eqnarray}
Then apply Baker's theorem to the absolute value of the right hand side
and adjust the resulting constant for ${ \frac{1} {2 \pi} }$. (Also
recall that $N \neq 0$ or $\pm 1$.)  
\end{proof}

\section{The main theorem and its variant}

\subsection{The main theorem and its proof}
We will prove Theorem~\ref{comp} by breaking it into several cases.  We
begin with the case where the place $v$ is finite.  For the sake of
precision, we will state when we need $\alpha$ to be in $K$ and when it
suffices that it be in $\Kbv$.  
\begin{prop}\label{prev}
  Let $(\zeta_n)_{n=1}^{\infty}$ be a sequence of distinct roots of
  unity, and write $x_n := \zeta_n + {\zeta_n}^{-1}$ for any $n \geq
  1$.  If $v$ is finite, then for any $\alpha \in \Kbv$, we have
\begin{equation}\label{finite}
\hhat_v(\alpha) = \lim_{n \rightarrow \infty}
{ \frac{1} { [ K(x_n) : K ] } } \sum_{\sigma:
  K(x_n)/K \hookrightarrow \overline{K}_v } \log | \sigma (x_n) - \alpha
|_v. 
\end{equation}
\end{prop}
\begin{proof}
  If $| \alpha |_v > 1$, then $| \sigma (x_n) - \alpha |_v = |\alpha
  |_v$. Thus, \eqref{finite} is immediate. Now, suppose that $| \alpha |_v \leq 1$. 
  Let $r < 1$ be a real
  number. Let $x_m$ and $x_n$ satisfy that $| x_m - \alpha |_v \leq r$
  and $| x_n - \alpha |_v \leq r$. Then observe
\begin{eqnarray}
r 
& \; \geq \; & | (x_m - \alpha) - (x_n - \alpha) |_v \nonumber\\
& \; = \; & | x_m - x_n |_v  \nonumber\\
& \; = \; & 
\Big | (\zeta_m - \zeta_n) - { \frac{\zeta_m - \zeta_n}{\zeta_m \zeta_n}} 
\Big |_v
\nonumber\\
& \; = \; & | \zeta_m - \zeta_n |_v \; | 1 - (\zeta_m \zeta_n)^{-1} |_v
\nonumber\\
& \; = \; & | 1 - \zeta_m^{-1} \zeta_n |_v \; 
| 1 - (\zeta_m \zeta_n)^{-1} |_v. \nonumber
\end{eqnarray}
Hence either $| 1 - \zeta_m^{-1} \zeta_n |_v \leq \sqrt r$ or
$| 1 - (\zeta_m \zeta_n)^{-1} |_v \leq \sqrt r$.  
In the first (resp.~second) case it follows that $\zeta_m^{-1} \zeta_n$ 
(resp.~$(\zeta_m \zeta_n)^{-1}$)
must have order equal to a power of the prime number 
$\in {\mathbb Z}$ lying below $v$, and  
that there are only finitely many choices for 
$\zeta_m^{-1} \zeta_n$ (resp. $(\zeta_m \zeta_n)^{-1}$) in the first
(resp.~second) case. Thus, for any real $r < 1$, 
there are only finitely many indices $n \geq 1$
such that $| x_n - \alpha |_v \leq r$, which immediately implies the
desired convergence in this case.
\end{proof}

We now treat the archimedean $v$ for which $\alpha$ is outside the Julia
set at $v$.
\begin{prop}
  Let $x_n$ be as in Proposition~\ref{prev}.  If $v$ is archimedean, then
  for any $\alpha \in \bC - [-2,2]$, we have
$$\hhat_v(\alpha) = \lim_{n \rightarrow \infty}
{ \frac{1} { [ K(x_n) : K ] } } \sum_{\sigma:
  K(x_n)/K \hookrightarrow \overline{K}_v } \log | \sigma (x_n) - \alpha
|_v. $$
\end{prop}
\begin{proof}
  From \eqref{gen}, we have
  $$
  \int_{\bC} \log |z - \alpha|_v \, d \mu_v (z)= \hhat_v(\alpha),$$
  where $\mu_v := \mu_{v, \varphi}$ is the invariant measure for 
$\varphi$ at $v$.  This
  measure is supported on $[-2,2]$, so if $g$ is a function on $\bC$
  that agrees with $\log |z - \alpha|$ on $[-2,2]$ we have
\begin{equation}\label{agree}
\int_{\bC} g(z) d \mu_v(z)= \hhat_v(\alpha).
\end{equation}
Let $\epsilon = \min_{w \in [-2,2]} |w - \alpha|$ (note that $\epsilon
\not= 0$ since $\alpha \notin [-2,2]$) and define $g(z)$ as 
$$ g(z) = \min \Big(\log \max ( | z - \alpha |, \epsilon ), \log
(|\alpha| + 2) \Big).$$
Then $g$ is continuous and bounded on all of $\bC$ and agrees with
$\log |z - \alpha|$ on $[-2,2]$.  By \cite[Theorem 1.1]{bilu}, we have
$$ \int_{\bC} g(z) \, d \mu_v(z) = \lim_{n \rightarrow \infty}
{ \frac{1} { [ K(x_n) : K ] } } \sum_{\sigma:
  K(x_n)/K \hookrightarrow \overline{K}_v } g(\sigma(x_n)).$$
Since all $x_n \in [-2,2]$, this finishes the proof, using \eqref{agree}.
\end{proof}

Now, we come to the most difficult case.  
\begin{prop}~\label{most}
Let $x_n$ be as in 
Proposition~\ref{prev}.
If $v | \infty$ and $\alpha \in [-2,2]$ is not preperiodic, then we have
$$\hhat_v(\alpha) = \lim_{n \rightarrow \infty} \frac{1}{ [ K(x_n) : K
  ] } \sum_{\sigma: K(x_n)/K \hookrightarrow \overline{K}_v } \log |
\sigma (x_n) - \alpha |_v. $$
\end{prop}
\begin{proof}
  We may assume that $\alpha \in K$.  
  If $v$ is archimedean and $\alpha \in K$ is
  in $[-2,2]$, then we have $\hhat_v(\alpha) = 0$.  This follows from
  the fact that $\varphi$ maps $[-2,2]$ to itself, so if $\alpha \in
  [-2,2]$, then $|\varphi^n(\alpha)|_v$ is bounded for all $n$, so
  $\hhat_v(\alpha) = 0$ by \eqref{local}.
  
  Note $|x|_v = |\tau (x)|$ for all $x \in \bC$, where $\tau: K(x)/K
  \hookrightarrow \bC$ is associated to $v$ and $| \cdot |$ is the
  usual absolute value on $\bC$.  To simplify our notation, we will
  fix one $v | \infty$, suppress $v$ in the notation of the absolute
  value, and use $| \cdot |$ according to this
  observation, i.e., without loss of generality we will prove this
  theorem for the place $v$ equal to the usual absolute value ($|z|
  = \sqrt{z \overline z}$, $z \in \bC$). However, we will keep $v$ in
  the notation of the local height $\hhat_v$ to avoid any confusion with
  the global height $\hhat$.
 
We may write
$
\alpha = e^{2 \pi i \theta_0} + e^{-2 \pi i \theta_0} 
= 2 \cos (2 \pi \theta_0),
$ 
where $\theta_0 \in (- { \frac{1} {2} }, { \frac{1} {2} }]$.  Note
$\alpha$ cannot be equal to $-2$, $2$, or 0 since we assume that $\alpha$
is not preperiodic.  
Note that 
$\int_{0}^{\epsilon} \log \big ( { \frac{t} {\epsilon} } \big ) dt =
-\epsilon$ for any $\epsilon > 0$.

Write 
\begin{eqnarray}
x & \; = \; & 
e^{2\pi i {\theta }} + e^{-2\pi i {\theta }} \;\; \;\; = \;
2 \cos (2\pi {\theta}); \;\;\;\;\;\; 
\textup{and} \nonumber\\
x_n & \; = \; & 
e^{2\pi i { \frac{a}{N} } } + e^{-2\pi i { \frac{a}{N} } } \; = \; 
2 \cos \Big (2\pi { \frac{a}{N} } \Big ) \nonumber
\end{eqnarray}
where $a$ and $N (\neq 0)$  are integers (depending on $n \geq 1$), and 
$\big | { \frac{a}{N} } \big | \leq 1$. 

We recall that $\hhat_v(\alpha) = 0$ since for any $\alpha$ in $[-2,2]$,
the quantity $|P_m^k(\alpha)|$ is bounded for all $k \geq 1$.  Thus, we have

\begin{equation}\label{delta}
\frac{1}{\pi} \int_{-2}^{2} { \frac{1}{ \sqrt
    {4-x^2} } } \log | x - \alpha | \ dx = 0.
\end{equation}
Hence it will suffice to show, for all $n \gg 1$, that the quantity
\begin{equation} \label{quant}
 \frac{1}{ [ K(x_n) : K ] } \Bigg |
    \sum_{\sigma: K(x_n)/K \hookrightarrow \bC} \log
      |\sigma(x_n) - \alpha| \Bigg |
\end{equation}
can be made sufficiently small.  

Fix $\epsilon > 0$.
By \eqref{delta}, we have
\begin{equation}\label{epsi}
\Bigg | \int_{\alpha - \delta}^{\alpha + \delta} { \frac{1}{ \sqrt
    {4-x^2} } } \log | x - \alpha | \ dx \Bigg | < \epsilon,
\end{equation}
i.e., sufficiently small for all sufficiently small $\delta > 0$. 

Let $g_\delta (z) = \log \max (|z - \alpha|, \delta)$.  By
\eqref{epsi} and the fact that $0 > g_\delta(x) > \log |x - \alpha|$
for $x \in [\alpha - \delta, \alpha + \delta]$, we see that
\begin{equation}\label{epsi2}
\Bigg | \int_{-2}^{2} { \frac{1}{ \sqrt
    {4-x^2} } } g_\delta(x) \ dx \Bigg | < \epsilon.
\end{equation}

By the equidistribution theorem of Baker/Rumely
(\cite{BREQUI}), Chambert-Loir (\cite{CL}) and Favre/
Rivera-Letelier
(\cite{FR2}), we see that for all sufficiently large $n$, the quantity
  \begin{equation}\label{ep3}
  \left| \Big( \frac{1}{\pi} \int_{-2}^{2} \frac{1}{ \sqrt {4-x^2}}
    g_\delta(x) \; dx \Big) - \Big( \frac{1}{ [ K(x_n) : K ] }
    \sum_{\sigma: K(x_n)/K \hookrightarrow \bC }
    g_\delta(\sigma(x_n)) \Big) \right| 
\end{equation}
is sufficiently small.  Thus, by \eqref{epsi2} it suffices to show
that
\begin{equation*}
\frac{1}{ [ K(x_n) : K ] } 
\Bigg |
\sum_{\sigma: K(x_n)/K
  \hookrightarrow \bC } \big ( \log | \sigma(x_n) - \alpha | -
g_\delta(\sigma(x_n)) \big ) \Bigg | 
\end{equation*}
is sufficiently small for all $n \gg 1$ and all sufficiently small
$\delta > 0$.  Since $\log | \sigma(x_n) - \alpha | =
g_\delta(\sigma(x_n))$ outside of $[\alpha-\delta, \alpha+\delta]$, it
in turn it suffices to show that
\begin{equation}\label{a}
\frac{1}{ [ K(x_n) : K ] } 
\Bigg |
\sum_{\substack{\sigma: K(x_n)/K
    \hookrightarrow \bC\\ \sigma(x_n) \in [\alpha - \delta, \alpha +
    \delta] }} 
\big ( \log | \sigma (x_n) - \alpha | - g_\delta(\sigma(x_n)) \big ) 
\Bigg |
\end{equation}
is sufficiently small for all $n \gg 1$ and all sufficiently
small $\delta > 0$.  

Now, when $\delta > 0$ is small and $x$ is in $[\alpha - \delta,
\alpha + \delta]$, we have $0 > g_\delta(x) \geq \log |x - \alpha|$
and the quantity \eqref{a} is bounded above by
\begin{equation}\label{b}
\frac{1}{ [ K(x_n) : K ] } 
\Bigg | 
\sum_{\substack{\sigma: K(x_n)/K
    \hookrightarrow \bC\\ \sigma(x_n) \in [\alpha - \delta, \alpha +
    \delta] }} 
\log | \sigma (x_n) - \alpha | 
\Bigg |.
\end{equation}
Hence, finally it suffices to show that \eqref{b} is sufficiently small
whenever $n$ is sufficiently large and $\delta > 0$ is
sufficiently small.  

If we choose $\delta >0$ sufficiently small, we may assume that
\begin{equation}\label{M2}
\min_{x \in [\alpha - \delta, \alpha + \delta]} \left( \frac{1}{\sqrt{4 -
    x^2}} \right) \geq \frac{1}{2} \max_{x \in [\alpha - \delta, \alpha
  + \delta]} \left( \frac{1}{\sqrt{4 - x^2}} \right). 
\end{equation}
We define $M$ as
\begin{equation}\label{M3}
M := \max_{x \in [\alpha - \delta, \alpha
  + \delta]} \left( \frac{1}{\sqrt{4 - x^2}} \right)
\end{equation}

Choose a large positive integer $D$.  For any $1 \leq i \leq D$
denote by $S_i$ the interval
$$[\alpha -
\delta + (i-1)(\delta/D), \alpha - \delta + i (\delta/D)].$$
Given any $n \gg 1$,
let $N_i := N_i (n)$ denote the number of $\sigma(x_n)$'s belonging to
$S_i$.

Note that $\log | \sigma (x_n) - \alpha | \leq 0$,
              whenever $\sigma (x_n)$ belongs to any of the
              $S_i$ $(1 \leq i \leq D)$.
For any $1 \leq i \leq D -1$, on $S_i$ we have
\begin{equation*}
\begin{split}
& \frac{1}{[K(x_n):K]} \biggl | \sum_{\substack{\sigma: K(x_n)/K 
\hookrightarrow \bC \\ \sigma(x_n)
    \in S_i}} \log | \sigma(x_n) - \alpha| \biggl | \\
& \leq  M
(\delta/D) \Big | \log |(D-i)
(\delta/D)| \Big | + O\left(\frac{1}{\sqrt{[K(x_n):K]}} \right) \Big | \log
  \big ( (D-i)\delta/D \big ) \Big | \\
& \quad \quad \text{(by Lemma~\ref{lem;counting} with $\gamma = 1/2$)}\\
& \leq 2 \Bigg | \int_{S_{i+1}} (M/2) \log |x - \alpha| \, dx \Bigg |
+ 
O\left(\frac{1}{\sqrt{[K(x_n):K]}} 
\right) \Big | \log \big ((D-i) \delta/D \big ) \Big |\\
& \leq 2 \Bigg | 
\int_{S_{i+1}} \frac{1}{\sqrt{4 - x^2}} \log |x - \alpha| \, dx \Bigg |
+ O\left(\frac{1}{\sqrt{[K(x_n):K]}}\right) \Big |\log \big ( (D-i) \delta/D
\big ) \Big | \\
& \ \ 
\quad \text{ (by \eqref{M2} and \eqref{M3}) }.  
\end{split}
\end{equation*}

Summing up over all $1 \leq i \leq D-1$ and applying \eqref{epsi} we
obtain
\begin{equation}\label{smaller}
\begin{split}
\frac{1}{[K(x_n):K]} \Bigg | \sum_{\substack{\sigma: K(x_n)/K 
\hookrightarrow \bC \\ \sigma(x_n)
    \in [\alpha - \delta, \alpha - (\delta/D)]}} & \log |
\sigma(x_n) - \alpha| \Bigg |\\
 & \leq  2 \epsilon +  \frac{1}{\sqrt{[K(x_n):K]}}
C_2 D \big ( \big | \log (\delta/D) \bigl | 
 + \log D \big ),
\end{split}
\end{equation}
for some constant $C_2 > 0$ independent of $n$ and $D$.

Similarly, we see that
\begin{equation}\label{bigger}
\begin{split}
\frac{1}{[K(x_n):K]} \Bigg | 
\sum_{\substack{\sigma: K(x_n)/K \hookrightarrow \bC \\ \sigma(x_n)
    \in [\alpha + (\delta/D), \alpha + \delta]}} & \log |
\sigma(x_n) - \alpha| \Bigg | \\
& \leq  2 \epsilon + \frac{1}{\sqrt{[K(x_n):K]}}
C_3  D \big ( \big| \log (\delta/D) \big| + \log D \big ),
\end{split}
\end{equation}
for some constant $C_3 > 0$ independent of $n$ and $D$.  Since
$|\log(1/D)|$ and $\log D$ grow more slowly than any power of $D$, 
we see that
quantities \eqref{smaller} and \eqref{bigger} can be made sufficiently
small when $D$ is large and $[K(x_n):K] \geq D^4$.    

Now, for all sufficiently small $\delta > 0$, we have
$$
0 \geq \log  | x_n - \alpha |  = \log
\big | 2 \cos \Big (2\pi { \frac{a}{N} } \Big ) - 
2 \cos (2\pi {\theta_0}) \big | \geq  \log
\Big | { \frac{a}{N} } - 
\theta_0 \Big | + O(1)
$$
for all $x_n \in [\alpha - \delta, \alpha + \delta]$.  When
$N$ is sufficiently large, Theorem~\ref{thm;baker} thus yields
$$
0 \geq \log | x_n - \alpha | \geq -C_4 {\log N} + O(1) 
$$
where $C_4 > 0$ is a constant independent of $n$.  This inequality
is true not only for $x_n$ itself, but also for all its $K$-Galois
conjugates that belong to $[\alpha - \delta, \alpha + \delta]$, i.e.,
after readjusting $C_4$ if necessary, we have
             \begin{equation}\label{baker-log}
             0 \geq \log | \sigma (x_n) - \alpha | \geq
             - C_4 \log N 
             \end{equation}
             for all $\sigma (x_n) \in [\alpha - \delta, \alpha +
             \delta]$, where $C_4 > 0$ is a constant independent of
             (all) $n \gg 1$.  Thus, it follows from \eqref{M} 
             (again with $\gamma
             = 1/4$) that we have
\begin{equation}\label{close}
\begin{split}
\frac{1}{[K(x_n):K]}  
\Bigg | 
\sum_{\substack{\sigma: K(x_n)/K \hookrightarrow \bC \\
\sigma(x_n) \in [\alpha - (\delta/D), \alpha + (\delta/D)]}} \log
      |\sigma(x_n) - \alpha | 
\Bigg | \\
\leq C_4 M (\delta/D) \log N +
    \frac{C_5 
\log N}{[K(x_n):K]^{1/2} }
\end{split}
\end{equation}
    where $C_5 > 0$ is a constant.

     Write $\phi$ for the Euler function, and suppose that $N
    \gg 1$. Note that
$$
[K(x_n) : K] \; \geq \; 
\frac{[\mathbb Q (x_n) : {\mathbb Q}] }
{[K : \mathbb Q]} \; = \; { \frac{\phi (N)} {[K : \mathbb Q]} }
$$
and $\phi (N) \geq \sqrt N$ (see \cite[page 267, Thm 327]{HW}), and hence that
                 $[K(x_n):K]^{ \frac{1}{2} } \gg \sqrt[4]{N}$.  
Now, let $D =
\lfloor \sqrt[4]{N} \rfloor $.  
(Note this choice of $D$ is compatible with that of $D$ in
\eqref{smaller} and \eqref{bigger}.)
Then, when $N$ is sufficiently large, the right-hand
sides of \eqref{smaller} and \eqref{bigger} are both sufficiently small 
and the right-hand side of \eqref{close} is also sufficiently
small.  Combining equations \eqref{smaller}, \eqref{bigger}, and
\eqref{close} we then obtain that
$$
\frac{1}{[K(x_n):K]} 
\Bigg | 
\sum_{\sigma: K(x_n)/K \hookrightarrow
    \bC } \log |\sigma(x_n) - \alpha | 
\Bigg | \;\; 
\textup{is sufficiently small}.$$
Thus, we must have
$ \lim_{n \to \infty} \sum_{\sigma: K(x_n)/K \hookrightarrow
    \bC } \log |\sigma(x_n) - \alpha | = 0,$
as desired.
\end{proof}

The proof of Theorem~\ref{comp} is now immediate since the
Propositions above cover all $v$ and all non-preperiodic 
$\alpha \in K$.  Now, we are
ready to prove our main theorem, Theorem~\ref{main}.  
\begin{proof}[Proof of Theorem~\ref{main}]
  Let $S$ be a finite set of places of $K$ that includes all the archimedean
  places.  After extending $S$ to a larger finite set if
                 necessary, which only makes the set ${\mathbb
    A^1}_{\varphi, \alpha, S}$ larger, we may assume that $S$ also
  contains all the places $v$ for which $|\alpha|_v > 1$.  Then for any $v
  \notin S$ and any preperiodic point $x_n$ we have 
\begin{equation}\label{zero}
 \log | \sigma (x_n) - \alpha |_v = 0 \; \; \text{for any embedding
   $\sigma: K(x_n)/K \lra \bKv$.}
\end{equation}

Assume that $(x_n)_{n=1}^\infty$ is an infinite nonrepeating
  sequence in ${\mathbb A^1}_{\varphi, \alpha, S}$.  Since we can
  interchange a limit with a finite sum, we have
\begin{equation}
\begin{split}
  & \frac{1}{[ K : {\mathbb Q}]} \hhat(\alpha) 
  = \sum_{v \in S} \lim_{n \rightarrow \infty}
  \frac{1}{[ K(x_n) : {\mathbb Q}]} \sum_{\sigma: K(x_n)/K
    \hookrightarrow \bKv } \log | \sigma (x_n) - \alpha |_v\\
 & \text {(by \eqref{global}, \eqref{zero}, and Thm.~\ref{comp})
  }\\
  & = \lim_{n \rightarrow \infty} \sum_{v \in S} \frac{1}{[ K(x_n) :
    {\mathbb Q}]} \sum_{\sigma: K(x_n)/K
    \hookrightarrow \bKv } \log | \sigma (x_n) - \alpha |_v
  \; \; \text{(switching sum and limit)} \\
  &= \lim_{n \rightarrow \infty} \sum_{\text{places $v$ of $K$}}
  \frac{1}{[ K(x_n) : {\mathbb Q}]} \sum_{\sigma: K(x_n)/K
    \hookrightarrow \bKv } \log | \sigma (x_n) - \alpha |_v
  \quad \text{(by \eqref{zero})}\\
  & = 0 \quad \quad \quad \text{(by the product formula)}.
\end{split}
\end{equation}
Since $\alpha$ is not preperiodic, however, we have $\hhat(\alpha)
> 0$.  Thus, we have a contradiction, so ${\mathbb A^1}_{\varphi,
  \alpha, S}$ must be finite.
\end{proof}

\subsection{A variant of the Chebyshev dynamical systems}

We look at different Chebyshev polynomials defined by the
following recursion formula: 
$$
Q_1 (z) := z, \;\;\; Q_2 (z) := z^2 + 2;  \;\; \textup{and}
$$ 
$$
Q_{m+1} (z) - Q_{m-1} (z) = z Q_{m} (z) \;\; \textup{for all} \;
m \geq 2. 
$$
The dynamical system induced by any of the $Q_{m}$ 
($m \geq 2$) on ${\mathbb A^{1}}$ (or
${\mathbb P^1}$) has properties similar to those 
for the Chebyshev dynamical systems, for instance:
\begin{enumerate}
\item[{\rm (i)}] The Julia set is equal to the interval
$[-2, 2]$ on the $y$-axis;

\item[{\rm (ii)}] The preperiodic points are (either $\infty$ or)
the points of type $\zeta - \zeta^{-1}$, where $\zeta$ is a root of
unity.

\item[{\rm (iii)}] The corresponding measures $\mu_v$ satisfy
\begin{eqnarray}
\int_{ {\mathbb P^1} (\mathbb C_v)} \log |z-\alpha|_v \; d\mu_v = 
\left\{
\begin{array}{ll}  \log \max \{ |\alpha|_v, \; 1 \}, & 
\text{if $v \not | \infty$}; \\
{ \frac{1}{\pi} } \int_{-2}^{2} 
{ \frac{1} { \sqrt {4-y^2} } }
\log | yi - \alpha |_v \; dy, & \text{otherwise}
\end{array} \right. \nonumber
\end{eqnarray}
where $\alpha \in K$ ($K$ a number field), $v$ is a place of $K$, and  
$dy$ is the usual Lebesgue measure on $[-2, 2]$. Note that the measure
$\mu_v$ ($v | \infty$) is supported on the interval $[-2, 2]$ on the 
$y$-axis.

\end{enumerate}

\noindent It is then easy to see that arguments similar to the above 
prove the following:

\begin{thm} Let $\psi$ be any of the $Q_m$ ($m \geq 2$). 
Let $K$ be a number field, and let $S$ be a finite set of
places of $K$, containing all the infinite ones. Suppose that 
$\alpha \in K$ is not of type $\zeta - \zeta^{-1}$ for any root of 
unity $\zeta$. Then the following set
$$ {\mathbb A^1}_{\psi, \alpha, S} :=
\{ z \in {\overline {\mathbb Q}} : z \; 
\textup{{\em{is $S$-integral with respect to}}}
\; \alpha \; \textup{{\em{and is}}} \; 
\textup{${\psi}$-{\em{preperiodic}}} \}
$$
is finite.
\end{thm}

\def\cprime{$'$} \def\cprime{$'$} \def\cprime{$'$}
\providecommand{\bysame}{\leavevmode\hbox to3em{\hrulefill}\thinspace}
\providecommand{\MR}{\relax\ifhmode\unskip\space\fi MR }
\providecommand{\MRhref}[2]{%
  \href{http://www.ams.org/mathscinet-getitem?mr=#1}{#2}
}
\providecommand{\href}[2]{#2}

\end{document}